[1994/12/01]

\documentclass[12pt,fleqn]{article}
\usepackage{amsfonts,amssymb,amsmath,amsthm}

\newtheorem{theorem}{Theorem}[section]

\newtheorem{proposition}[theorem]{Proposition}

\newtheorem{corollary}[theorem]{Corollary}

\newtheorem{remark}[theorem]{Remark}
\newtheorem{conj}[theorem]{Conjecture}
\theoremstyle{definition}

\usepackage{ulem}
\def\a{\mathfrak{a}}

\def\g{\mathfrak{g}}
\def\k{\mathfrak{k}}
\def\p{\mathfrak{p}}

\def\so{\mathfrak{so}}
\def\sp{\mathfrak{sp}}

\def\R{{\bf R}}

\def\diag{\mathop{\hbox{diag}}}
\def\Ad{\mathop{\hbox{\small Ad}}}

\def\to{\rightarrow}

\begin{document}

\begin{center}
Sharp Estimates of Radial Dunkl and Heat Kernels in the Complex Case $A_n$\\
P.{} Graczyk\footnote{LAREMA, UFR Sciences, Universit\'e d'Angers, 2 bd Lavoisier, 49045 Angers cedex 01, France,piotr.graczyk@univ-angers.fr} and P.{} Sawyer\footnote{Department of Mathematics and Computer Science, Laurentian University, Sudbury, Canada P3E 2C6, psawyer@laurentian.ca}
\end{center}

\begin{abstract} 
In this article, we consider the radial Dunkl geometric case $k=1$ corresponding   
to flat  Riemannian symmetric spaces in the complex case
and we prove exact estimates for the positive valued Dunkl kernel and for the radial heat kernel.\\~\\
Dans cet article, nous consid\'erons le cas g\'eom\'etrique radial de Dunkl $ k = 1 $ correspondant
aux espaces sym\'etriques riemanniens plats dans le cas complexe
et nous prouvons des estimations exactes pour le noyau de Dunkl \`a valeur positive et pour le noyau de chaleur radial.
\end{abstract}

\section*{Thanks}
The authors thank Laurentian University, Sudbury and  the D\'efimaths programme of the R\'egion des Pays de la Loire for their financial support.

\section{Introduction  and notations}

Finding good estimates of  Dunkl heat kernels is a challenging and important subject, developed recently in \cite{AnkerDH}.  Establishing estimates of the heat kernels is equivalent to estimating the Dunkl kernel as demonstrated by equation \eqref{heatSpher} below.

In this paper we prove exact estimates in the $W$-radial Dunkl geometric case of multiplicity $k=1$, corresponding to Cartan motion groups and flat Riemannian symmetric spaces with the ambient group complex $G$, the Weyl group $W$ and the root system $A_n$.

We study for the first time the non-centered heat kernel, denoted $p_t^W(X,Y)$, on Riemannian symmetric spaces and we provide its sharp estimates. Exact estimates were obtained in \cite{Anker1} in the centered case  $Y=0$ for all Riemannian symmetric spaces.

We provide exact estimates for the spherical functions  $\psi_\lambda(X)$ in the two variables $X,\lambda$  when $\lambda$ is real and, consequently, for the heat kernel
$ p_t^W(X,Y)$ in the three variables $t,X,Y$.

We recall here some basic terminology and facts about symmetric spaces associated to Cartan motion groups.

Let $G$ be a semisimple Lie group and let $\g=\k\oplus\p$ be the Cartan decomposition of $G$. 
We recall the definition of the Cartan motion group and the flat symmetric space associated with the semisimple Lie group $G$ with maximal 
compact subgroup $K$. The Cartan motion group is the semi-direct product $G_0=K\rtimes \p$ where the multiplication is defined by $(k_1,X_1)\cdot(k_2,X_2)=(k_1\,k_2,\Ad(k_1)(X_2)+X_1)$. The associated flat symmetric space is then $M=\p\simeq G_0/K$ (the action of $G_0$ on $\p$ is given by $(k,X)\cdot Y=\Ad(k)(Y)+X$).

The spherical functions for the symmetric space $M$ are then given by
\begin{equation*}
\psi_\lambda(X)=\int_K\,e^{\lambda(\Ad(k)(X))}\,dk
\end{equation*}
where $\lambda$ is a complex linear functional on $\a\subset \p$, a Cartan subalgebra of the Lie algebra of $G$.  To extend $\lambda$ to $X\in\Ad(K)\a=\p$, one uses $\lambda(X)=\lambda(\pi_\a(X))$ where $\pi_\a$ is the orthogonal projection with respect to the Killing form (denoted throughout this paper by $\langle \cdot,\cdot\rangle$).  
Note that in \cite{Helgason1, Helgason2, Helgason3}, $\lambda$ is replaced by $i\,\lambda$.

Throughout this paper, we usually assume that $G$ is a semisimple complex Lie group. The complex root systems are respectively $A_n$ for $n\geq 1$ (where $\p$ consists 
of the $n\times n$ hermitian matrices with trace 0), $B_n$ for $n\geq 2$ (where $\p=i\,\so(2\,n+1)$), $C_n$ for $n\geq 3$ (where $\p=i\,\sp(n)$) 
and $D_n$ for $n\geq 4$ (where $\p=i\,\so(2\,n)$) for the classical cases and the exceptional root systems $E_6$, $E_7$, $E_8$, $F_4$ and $G_2$.

The radial heat kernel is considered with respect to the invariant measure $\mu(dY)=\pi^2(Y)\,dY$ on $M$,
where $\pi(Y)=\prod_{\alpha>0} \alpha(Y)$.

Note also that in the curved case $M_0=G/K$,  the spherical functions for the symmetric space $M_0$ are then given by
\begin{equation*}
\phi_\lambda(e^X)=\int_K\,e^{(\lambda-\rho)H(e^X\,k)}\,dk
\end{equation*}
where $\rho$ is the half-sum of the roots counted with their multiplicities and $H(g)$ is the abelian component in the Iwasawa decomposition of $g$: $g=k\,e^{H(g)}\,n$.

\section{Estimates of spherical functions and of the heat kernel}\label{Spherical}

We will be developing a sharp estimate for the spherical function $\psi_\lambda(X)$. We  introduce the following useful convention. We will write
\begin{align*}
f(t,X,\lambda)\asymp g(t,X,\lambda)
\end{align*}
in a given domain of $f$ and $g$ if there exists constants $C_1>0$ and $C_2>0$
independent of $t$, $X$ and $\lambda$ such that
$
C_1\,f(t,X,\lambda)\leq  g(t,X,\lambda)\leq  C_2 \, g(t,X,\lambda)
$
in the domain of consideration.

We conjecture the following global estimate for the spherical function in the complex case.
\begin{conj}\label{conj}
On flat Riemannian symmetric spaces with complex group $G$,
we have
\begin{equation*}
\psi_\lambda(X)
\asymp \frac{e^{\langle \lambda,X \rangle}}{\prod_{\alpha > 0 \ } (1+ \alpha(\lambda) \alpha(X))},\qquad \lambda\in{\overline{\a}^+}, X\in
 {\overline{\a}^+}.
\end{equation*}
\end{conj}

\begin{remark}
Recall that, denoting  $\delta(X) =\prod_{\alpha>0}\sinh^2 \alpha(X)$, we have
\begin{align}
\phi_\lambda(e^X)&=\frac{\pi(X)}{\delta^{1/2}(X)}\,\psi_\lambda(X).\label{phipsi}
\end{align}

Since $ \delta^{1/2}(X) \asymp e^{\rho(X)}\,\pi(X)/\prod_{\alpha>0}\,(1+\alpha(X))$ in the complex case, Conjecture \ref{conj} therefore becomes
\begin{align}\label{Ccurved}
\phi_\lambda(e^X)&\asymp
e^{(\lambda-\rho)(X)}\,\prod_{\alpha>0}\,\frac{1+\alpha(X)}{1+ \alpha(\lambda) \alpha(X)}   
\end{align}
in the curved complex case.

Let us compare the estimate
\eqref{Ccurved}
we conjecture for $\phi_\lambda$ with the one obtained in \cite{Narayanana}, cf.{} also \cite{Schapira}. The estimates in  \cite{Narayanana}
 apply in all the generality of hypergeometric functions of Heckman and Opdam. 
The authors show that there exists constants $C_1(\lambda)>0$, $C_2(\lambda)>0$ such that
\begin{align*}
C_1(\lambda)\,e^{(\lambda-\rho)(X)}
\,\prod_{\genfrac{}{}{0pt}{}{\alpha>0,}{\alpha(\lambda)=0}}\,(1+\alpha(X))
\leq\phi_\lambda(e^X)\leq C_2(\lambda)\,e^{(\lambda-\rho)(X)}
\,\prod_{\genfrac{}{}{0pt}{}{\alpha>0,}{\alpha(\lambda)=0}}\,(1+\alpha(X)).
\end{align*}
Given \eqref{phipsi}, corresponding estimates clearly also hold in the flat case for $\psi_\lambda(X)$.
The interest of our result, in the case $A_n$, lies in the fact that our estimate is universal in both $\lambda$ and $X$.
\end{remark}
The results of  \cite{Narayanana,  Schapira} and our estimates in the $A_n$ case strongly suggest that the  Conjecture  \ref{conj} is true for any complex root system.

Note that asymptotics of $\psi_\lambda(t\,X)$ when $\lambda$ and $X$ are singular and $t\to\infty$
were proven in \cite{PGPS1} for all classical complex root systems and  the systems $F_4$ and $G_2$. \\

Consider the relationship between the Dunkl kernel $E_k(X,Y)$ and the Dunkl heat kernel $p_t(X,Y)$, as given in \cite[Lemma 4.5]{R0}  
\begin{align} \label{heatSpher}
p_t(X,Y)&=
\frac1{2^{\gamma+d/2} c_k}
\,t^{-\frac{d}{2}-\gamma}
\,e^{\frac{-|X|^2-|Y|^2}{4t}}\,E_k\left(X,\frac{Y}{2t}\right),
\end{align}
where $\gamma$ is the number of positive roots and the constant  $c_k$ is the Macdonald--Mehta--Selberg integral.  
The formula \eqref{heatSpher} remains true for the $W$-invariant kernels  $p^W_t$  and   $E^W$. In the geometric cases $k=\frac12, 1$ and $2$, by \cite{dJ}, the $W$-invariant formula \eqref{heatSpher} translates  
in a similar relationship between the spherical function $\psi_\lambda$ and the heat kernel $p_t^W(X,Y)$:
\begin{align}\label{heatSpher1}
p_t^W(X,Y)&=\frac1{2^{\gamma+d/2} c_k} \,t^{-\frac{d}{2}-\gamma} \,e^{\frac{-|X|^2-|Y|^2}{4t}}\,\psi_X\left(\frac{Y}{2t}\right).
\end{align}
A simple direct proof of \eqref{heatSpher1} for $k=1$ is given in \cite[Remark 2.9]{PGPS1}.

Equation \eqref{heatSpher1} and Conjecture \ref{conj} bring us to an equivalent conjecture for the heat kernel $p_t^W(X,Y)$.

\begin{conj} \label{heatconj}
We have
\begin{equation*}
p_t^W(X,Y) \asymp t^{-\frac{d}{2}}
\,\frac{e^{\frac{-|X-Y|^2}{4t}}}{\prod_{\alpha>0}\,(t+\alpha(X)\,\alpha(Y))}.
\end{equation*}
\end{conj}

Consider also the relationship between the heat kernel $p_t^W(X,Y)$ and the heat kernel $\tilde{p}^W_t(X,Y)$ in the curved case.  We have 
\begin{align}\label{CurvedFlat}
\tilde{p}^W_t(X,Y)&= e^{-|\rho|^2 t}\, \frac {\pi(X)\,\pi(Y)}{\delta^{1/2}(X)\,\delta^{1/2}(Y)} \,p_t^W(X,Y).
\end{align}
This relation follows directly from the fact that the respective radial Laplacians and radial measures are $\pi^{-1}\,L_\a\circ \pi$ and $\pi(X)\,dX$ in the flat case and $\delta^{-1/2}\,(L_\a-|\rho|^2)\circ \delta^{1/2}$ and $\delta(X)\,dX$  in the curved case ($L_\a$ stands for the Euclidean Laplacian on $\a$).

In the curved complex case, Conjecture \ref{heatconj} becomes 
\begin{align*}
\tilde{p}^W_t(X,Y)&\asymp 
e^{-|\rho|^2 t} t^{-\frac{d}{2}}\,e^{-\rho(X+Y)}
\,\prod_{\alpha>0}\,\frac{(1+\alpha(X))\,(1+\alpha(Y))}{(t+\alpha(X)\,\alpha(Y)}\,e^{\frac{-|X-Y|^2}{4t}}.
\end{align*}

\begin{remark}
In  \cite{PGPSTL}, sharp estimates of $W$-invariant Poisson and Newton kernels in the complex Dunkl case were obtained,
by exploiting the method of construction of these  $W$-invariant kernels by alternating sums.
When a root system $\Sigma$ acts in $\R^d$, the sharp estimates of \cite{PGPSTL} have the common form
\begin{equation}\label{POISSON}
{\mathcal K}^W(X,Y)
\asymp \frac{{\mathcal K}^{\R^d}(X,Y)}{\prod_{\alpha > 0 \ } (|X-Y|^2+ \alpha(X)\, \alpha(Y))},\qquad  X,Y\in{\overline{\a}^+},
\end{equation}
where ${\mathcal K}^W(X,Y)$  is the $W$-invariant  kernel in Dunkl setting and ${\mathcal K}^{\R^d}(X,Y)$ is the classical kernel on  $\R^d$.
Let us observe a common pattern in the appearance of the classical kernels ${\mathcal K}^{\R^d}$ and of products of roots $\alpha(X)\,\alpha(Y)$
in formulas \eqref{POISSON} and of the  Fourier kernel $e^{\langle \lambda,X\rangle}$ and the classical Gaussian heat kernel and 
of  products $\alpha(\lambda)\alpha(X)$ in 
the estimates given in Conjecture \ref{conj} and Conjecture \ref{heatconj}.
\end{remark}

\subsection{Proof of Conjecture \ref{conj} in some cases}
We start with a practical result.
\begin{proposition}\label{XWY}
Let $\alpha_i$ be the simple roots and let $A_{\alpha_i}$ be such that $\langle X,A_{\alpha_i}\rangle=\alpha_i(X)$ for $X\in\a$.
Suppose 
$X\in\a^+$ and $w\in W\setminus\{id\}$. Then we have
\begin{align}\label{CL}
Y-w\,Y=\sum_{i=1}^r \,2\,\frac{a_i^w(Y)}{|\alpha_i|^2}\,A_{\alpha_i}
\end{align}
where $a_i^w$ is a linear combination of positive simple roots with non-negative integer coefficients for each $i$.
\end{proposition}

\begin{proof}
Refer to \cite{PGPSTL}.
\end{proof}

\begin{remark}\label{M}
Note that $a_i^w(Y)/|\alpha_i|^2$ is bounded by $C\,\max_k\,|\alpha_k(Y)|$ where $C$ is a constant depending only on $w\in W$ and, ultimately, on $W$.
\end{remark}

\begin{corollary}\label{alpha k in ak}
Let $Y\in \overline{\a^+}$ and $w\in W$. Consider the decomposition \eqref{CL} of $Y-wY$. If $a_k^w(Y)\not= 0$ then $\alpha_k$ appears in $a_k^w$, i.e. 
$a_k^w=\sum_{i=1}^r n_i \alpha_i$ with $n_k>0$.
\end{corollary}

\begin{proof}
Refer to \cite{PGPSTL}.
\end{proof}

\begin{proposition}\label{E1}
Let $\delta>0$.
Suppose $\alpha_i(\lambda)\,\alpha_j(X)\le\delta$ for all $i$, $j$.  Then $\psi_{\lambda}(X) \asymp e^{\lambda(X)}$ (the constants involved only depend on $\delta$).
\end{proposition}

\begin{proof}
Let $K(X,Y)$ be the kernel of the Abel transform. Recall that $K(X,Y)\,dY$ is a probability measure supported on $C(X)$, the convex envelope of the orbit $W\cdot X$.
Notice  that
\begin{equation}\label{fram}
e^{w_{min}\,\lambda(X)}\le \psi_\lambda(X)=\int_{C(X)}\,e^{\lambda(Y)}\,K(X,Y)\,dY \le e^{\lambda(X)}
\end{equation}
where
$w_{min}$ is the element of the Weyl group giving the minimum value of $w\,\lambda(X)$.
Now, using Proposition \ref{XWY} and Remark \ref{M} with $Y=\lambda$, we see that for any $w\in W $
\begin{align*}
e^{\lambda(X)}&\geq e^{w\,\lambda(X)}=e^{\langle w\lambda-\lambda,X\rangle}\,e^{\langle \lambda,X\rangle}
=\prod_{i=1}^r\,e^{-2\,\frac{a_i^w(\lambda)}{|\alpha_i|^2}\,\alpha_i(X)}\,e^{\langle \lambda,X\rangle}
\\&
\geq \prod_{i=1}^r\,e^{-2\,C\,(\max_k\,\alpha_k(\lambda))\,\alpha_i(X)}\,e^{\langle \lambda,X\rangle}
\geq \prod_{i=1}^r\,e^{-2\,C\,\delta}\,e^{\langle \lambda,X\rangle}.
\end{align*}
\end{proof}

\begin{remark}
This case and this method apply for any radial Dunkl case; it suffices to replace
$K(X,Y)\,dY$  by the so-called R\"osler measure $\mu_X(dY)$ in the integral in \eqref{fram}, see \cite{R1}.
\end{remark}

\begin{proposition}\label{new spherical}
A spherical function $\psi_{\lambda}(X)$ on $M$ is given by the formula
\begin{equation}\label{spherical fn 2}
\psi_{\lambda}(X)=\frac{\pi(\rho)}{2^\gamma\pi(\lambda)\,\pi(X)}\,\sum_{w\in W}\,{\epsilon(w)} e^{\langle w \lambda, X\rangle},
\end{equation}
where  $\rho=\frac{1}{2} \,\sum_{\alpha\in\Sigma^+}\,m_\alpha \alpha=
\sum_{\alpha\in\Sigma^+} \,\alpha$ and $\gamma=|\Sigma^+|$ is the number of positive roots (refer to \cite[Chap.{} IV, Proposition 4.8 and Chap.{} II, Theorem 5.35]{Helgason3}).
\end{proposition}

\begin{proposition}\label{BIG}
Suppose $\alpha(\lambda)\,\alpha(X)\ge (\log |W|)/2$ for all $\alpha>0$.
Then $$\psi_\lambda(X) \asymp
\displaystyle\frac{e^{\lambda(X)}}{\pi(\lambda)\pi(X)}.$$ 

We are assuming here that $|\alpha_i|\geq 1$ for each $i$.
\end{proposition}

\begin{proof}
Suppose $w\in W$ is not the identity.  In that case, $a_i^w(\lambda)$ is not equal to 0 for some $i$.  By  Proposition \ref{XWY} with $y=\lambda$ and
Corollary \ref{alpha k in ak}, $\lambda(X)-w\,\lambda(X)\geq 2\,a_i^{w}(\lambda)\,\alpha_i(X)/|\alpha_i|^2\geq 2\,\alpha_i(\lambda)\,\alpha_i(X)\geq \log |W|$. Each term
 $e^{\langle w \lambda, X\rangle}$
in the alternating sum \eqref{spherical fn 2} corresponding to $w\not=\hbox{id}$ is bounded by $e^{-\log |W|}\,e^{\lambda(X)}=e^{\lambda(X)}/|W|$.  Hence, since only half the terms in the sum are negative,
\begin{align*}
|W|\,e^{\lambda(X)}\geq \sum_{w\in W}\,{\epsilon(w)} e^{\langle w \lambda, X\rangle}
\geq e^{\lambda(X)}-\frac{|W|}{2}\,e^{\lambda(X)}/|W|=\frac{1}{2}\,e^{\lambda(X)}.
\end{align*}
\end{proof}

\section{The conjecture in the case of the root system $A_n$}

We will prove the conjecture in the case of the root system of type $A$.
\begin{theorem}\label{main}
In the case of the root system of type $A_n$ in the complex case, we have
\begin{equation}\label{MAIN}
\psi_\lambda(e^X)\asymp \frac{e^{\langle \lambda,X \rangle}}{\prod_{i<j}\,(1+(\lambda_i-\lambda_j)\,(x_i-x_j))},\qquad \lambda,~ X\in\overline{\a}^+.
\end{equation}
\end{theorem}

\begin{corollary}\label{heatAn}
\begin{align*}
\phi_\lambda(e^X)&\asymp  e^{(\lambda-\rho)(X)}\,\prod_{i<j}\,\frac{1+x_i-x_j}{1+(x_i-x_j)\,(\lambda_i-\lambda_j)},\\
p_t^W(X,Y)&\asymp t^{-\frac{d}{2}}
\,\frac{e^{\frac{-|X-Y|^2}{4t}}}{\prod_{i<j}\,(t+(x_i-x_j)\,(y_i-y_j))},\\
\tilde{p}^W_t(X,Y)&\asymp 
e^{-|\rho|^2 t}\, t^{-\frac{d}{2}}\,e^{-\rho(X+Y)}
\,\prod_{i<j}\,\frac{(1+x_i-x_j)\,(1+y_i-y_j)}{(t+(x_i-x_j)\,(y_i-y_j))}\,e^{\frac{-|X-Y|^2}{4t}}.
\end{align*}
\end{corollary}

We recall (refer to \cite{Sawyer}) the following iterative formula for the spherical functions of type $A$ in the complex case.  Here we do not assume that the elements of the Lie algebra have trace 0. Here the Cartan subalgebra $\a$ for the root system $A_{n-1}$ is isomorphic to $\R^n$. For $\lambda,X\in{\overline{\a}^+} \subset \R^n$,
we have
\begin{align}
\psi_\lambda(e^X)&=e^{\lambda(X)}\ \hbox{if $n=1$ and}\nonumber\\
\psi_\lambda(e^X)&=(n-1)!\,e^{\lambda_n\,\sum_{k=1}^n\,x_k}\,(\prod_{i<j}\,(x_i-x_j))^{-1}\,\int_{x_n}^{x_{n-1}}\,\cdots\,\int_{x_2}^{x_1}
\,\psi_{\lambda_0}(e^Y)\label{iter}
\\&\qquad\qquad\qquad\qquad\qquad\qquad\qquad\qquad
\,\prod_{i<j<n}\,(y_i-y_j)\,dy_1\cdots dy_{n-1}\nonumber
\end{align}
where $\lambda_0(U)=\sum_{k=1}^{n-1}\,(\lambda_k-\lambda_n)\,u_k$.

\begin{remark}
Formula \eqref{iter} represents the action of the root system $A_{n-1}$ on $\R^n$.  If we assume $\sum_{k=1}^n\,x_k=0=\sum_{k=1}^n\,\lambda_k$, we have then the action of the root system $A_{n-1}$ on $\R^{n-1}$.  We can also consider the action of $A_{n-1}$ on any $\R^m$ with $m\geq n-1$ by considering formula \eqref{spherical fn 2} and deciding on which entries $x_k$, the Weyl group $W=S_n$ acts.  These considerations do not affect the conclusion of Theorem \ref{main}.
\end{remark}

\subsection{Approximate factorization for $A_n$}

Before proving the conjecture in the case $A_n$, we will prove an interesting ``factorization''.

\begin{proposition}\label{iteriter}
For $n\geq 1$, consider the root system $A_n$ on $\R^{n+1}$.
Let $\lambda,X\in{\overline{\a}^+}\subset \R^{n+1}$ and $X'=[X_1,\ldots,X_n].$
Define
\begin{align*}
I^{(n)}=I^{(n)}(\lambda;X)&=\int_{x_{n+1}}^{x_n} \int_{x_n}^{x_{n-1}}\,\cdots\,\int_{x_3}^{x_2}\,\int_{x_2}^{x_1}
\,e^{-\lambda_0(X'-Y)}
\\&\qquad
\,\prod_{i<j<n}\,\frac{(y_i-y_j)\,(\lambda_i-\lambda_j)}
{1+(y_i-y_j)\,(\lambda_i-\lambda_j)}\,dy_1\,dy_2\,\cdots\,dy_n.
\end{align*}
Then the following approximate factorization holds
\begin{equation}\label{factorI4}
I^{(n)}\asymp \prod_{k=1}^n\,I^{(n)}_k
\end{equation}
where
\begin{align*}
 I_1^{(n)}&=\int_{x_2}^{x_1}\,e^{-(\lambda_1-\lambda_{n+1})\,(x_1-y_1)}\,dy_1\ \hbox{and}\\
 I_k^{(n)}&=\int_{x_{k+1}}^{x_k}\,e^{-(\lambda_k-\lambda_{n+1})\,(x_k-y_k)}
\,\prod_{j=1}^{k-1}\,\frac{(x_j-y_k)\,(\lambda_j-\lambda_k)}{1+(x_j-y_k)\,(\lambda_j-\lambda_k)}
\,dy_k\ \hbox{for $1<k\leq n$}.
\end{align*}
\end{proposition}
\begin{proof}
Since $u/(1+u)$ is an increasing function, we clearly have
\begin{align*}
I^{(n)}&\leq\int_{x_{n+1}}^{x_n} \int_{x_n}^{x_{n-1}}\,\cdots\,\int_{x_3}^{x_2}\,\int_{x_2}^{x_1}
\,e^{-\lambda_0(X'-Y)}\,\prod_{i<j<n}\,\frac{(x_i-y_j)\,(\lambda_i-\lambda_j)}
{1+(x_i-y_j)\,(\lambda_i-\lambda_j)}\,dy_1\,dy_2\,\cdots\,dy_n.
\end{align*}
  On the other hand,
\begin{align*}
I^{(n)}&\geq\int_{(x_n+x_{n+1})/2}^{x_n} \int_{(x_{n-1}+x_n)/2}^{x_{n-1}}\,\cdots\,\int_{(x_2+x_3)/2}^{x_2}\,\int_{(x_1+x_2)/2}^{x_1}
\,e^{-\lambda_0(X'-Y)}
\\&
\,\prod_{i<j<n}\,\frac{(y_i-y_j)\,(\lambda_i-\lambda_j)}
{1+(y_i-y_j)\,(\lambda_i-\lambda_j)}\,dy_1\,dy_2\,\cdots\,dy_n\\
&\geq\int_{(x_n+x_{n+1})/2}^{x_n} \int_{(x_{n-1}+x_n)/2}^{x_{n-1}}\,\cdots\,\int_{(x_2+x_3)/2}^{x_2}\,\int_{(x_1+x_2)/2}^{x_1}
\,e^{-\lambda_0(X'-Y)}
\\&
\,\prod_{i<j<n}\,\frac{((x_i+x_{i+1})/2-y_j)\,(\lambda_i-\lambda_j)}
{1+((x_i+x_{i+1})/2-y_j)\,(\lambda_i-\lambda_j)}\,dy_1\,dy_2\,\cdots\,dy_n\\
&\asymp \int_{(x_n+x_{n+1})/2}^{x_n} \int_{(x_{n-1}+x_n)/2}^{x_{n-1}}\,\cdots\,\int_{(x_2+x_3)/2}^{x_2}\,\int_{(x_1+x_2)/2}^{x_1}
\,e^{-\lambda_0(X'-Y)}
\\&\qquad
\,\prod_{i<j<n}\,\frac{(x_i-y_j)\,(\lambda_i-\lambda_j)}
{1+(x_i-y_j)\,(\lambda_i-\lambda_j)}\,dy_1\,dy_2\,\cdots\,dy_n=\\
&
=\prod_{k=1}^n
\int_{(x_k+x_{k+1})/2}^{x_{k}}\,e^{-(\lambda_k-\lambda_{n+1})\,(x_k-y_k)}
\,\prod_{j=1}^{k-1}\,\frac{(x_j-y_k)\,(\lambda_j-\lambda_k)}{1+(x_j-y_k)\,(\lambda_j-\lambda_k)}
\,dy_k
=\prod_{k=1}^n\,A^{(n)}_k
\end{align*}
since 
\begin{align*}
\frac{((x_i+x_{i+1})/2-y_j)\,(\lambda_i-\lambda_j)}{1+((x_i+x_{i+1})/2-y_j)\,(\lambda_i-\lambda_j)}
\leq \frac{(x_i-y_j)\,(\lambda_i-\lambda_j)}{1+(x_i-y_j)\,(\lambda_i-\lambda_j)}
\end{align*}
while
\begin{align*}
\frac{((x_i+x_{i+1})/2-y_j)\,(\lambda_i-\lambda_j)}{1+((x_i+x_{i+1})/2-y_j)\,(\lambda_i-\lambda_j)}
\geq\frac{((x_i-y_j)/2\,(\lambda_i-\lambda_j)}{1+(x_i-y_j)/2\,(\lambda_i-\lambda_j)}
\geq \frac{1}{2}\, \frac{(x_i-y_j)\,(\lambda_i-\lambda_j)}{1+(x_i-y_j)\,(\lambda_i-\lambda_j)}.
\end{align*}

Now, let 
\begin{align*}
B^{(n)}_k&=\int_{x_{k+1}}^{(x_k+x_{k+1})/2}\,e^{-(\lambda_k-\lambda_{n+1})\,(x_k-y_k)}
\,\prod_{j=1}^{k-1}\,\frac{(x_j-y_k)\,(\lambda_j-\lambda_k)}{1+(x_j-y_k)\,(\lambda_j-\lambda_k)}
\,dy_k
\end{align*}
and note that $I^{(n)}_k=A^{(n)}_k+B^{(n)}_k$.

Now, using the change of variable $2w=x_k-y_k$, we have
\begin{align*}
B^{(n)}_k&=2\,\int_{(x_k-x_{k+1})/4}^{(x_k-x_{k+1})/2}\,e^{-2\,(\lambda_k-\lambda_{n+1})\,w}
\,\prod_{j=1}^{k-1}\,\frac{(x_j-x_k+2\,w)\,(\lambda_j-\lambda_k)}{1+(x_j-x_k+2\,w)\,(\lambda_j-\lambda_k)}
\,dw\\
&\leq 4\,\int_{(x_k-x_{k+1})/4}^{(x_k-x_{k+1})/2}\,e^{-2\,(\lambda_k-\lambda_{n+1})\,w}
\,\prod_{j=1}^{k-1}\,\frac{(x_j-x_k+w)\,(\lambda_j-\lambda_k)}{1+(x_j-x_k+w)\,(\lambda_j-\lambda_k)}
\,dw\\
&\leq 4\,\int_{0}^{(x_k-x_{k+1})/2}\,e^{-(\lambda_k-\lambda_{n+1})\,w}
\,\prod_{j=1}^{k-1}\,\frac{(x_j-x_k+w)\,(\lambda_j-\lambda_k)}{1+(x_j-x_k+w)\,(\lambda_j-\lambda_k)}
\,dw=4\,A^{(n)}_k,
\end{align*}
where the last equality comes from the change of variable $w=x_k-y_k$ in the expression for $A^{(n)}_{k}$.
Therefore $I^{(n)}_k=A^{(n)}_k+B^{(n)}_k\leq 5\,A^{(n)}_k$.  The result follows.
\end{proof}

The next proposition gives an inductive way of estimating $I^{(n+1)}$, knowing $I^{(n)}$ and $I^{(n-1)}$.

\begin{proposition}\label{recursive}
Consider the root system $A_{n+1}$ on $\R^{n+2}$. Let
$\lambda,X\in{\overline{\a}^+}\subset \R^{n+2}$.
Assume $\alpha_1(X)\geq \alpha_{n+1}(X)$.  Then
\begin{align*}
I^{(n+1)}(\lambda;X)&\asymp 
 I^{(n)}(\lambda_1  ,\ldots ,\lambda_n,\lambda_{n+2};x_1,\ldots,x_{n+1})
\frac{(x_1-x_{n+1})(\lambda_1-\lambda_{n+1}) }{1+(x_1-x_{n+1})(\lambda_1-\lambda_{n+1}) }
 \\&\qquad
\frac{I^{(n)}(\lambda_2  ,\ldots ,\lambda_{n+1},\lambda_{n+2};x_2,\ldots,x_{n+2})}
{I^{(n-1)}(\lambda_2  ,\ldots ,\lambda_{n},\lambda_{n+2};x_2,\ldots,x_{n+1})}.
\end{align*}
\end{proposition}
	\begin{proof}

We start with an outline of  the proof.
\begin{itemize}
\item[(i)] $I^{(n+1)}$ is estimated by a product of $n+1$ factors $I^{(n+1)}_k(\lambda;X)$.
\item[(ii)] The product of the first $n$ factors $I^{(n+1)}_1(\lambda;X)$, \ldots, $I^{(n+1)}_n(\lambda;X)$
give an estimate of the term $I^{(n)}(\lambda_1,\dots,\lambda_n,\lambda_{n+2};X')$ by Proposition \ref{iteriter}.
\item[ (iii)] In the last factor $I^{(n+1)}_{n+1}(\lambda;X)$, we ``draw off''
one term from under the integral, using the additional hypothesis
$\alpha_1(X)\geq \alpha_{n+1}(X)$.  The remaining integral corresponds to $I^{(n)}_n(\lambda_2,\dots,\lambda_{n+2};x_2,\dots,x_{n+2})$.
\item[(iv)] The last factor $I^{(n)}_n$ of $I^{(n)}$  is estimated by $I^{(n)}/I^{(n-1)}$, up to a change of variables (we re-use the idea of (ii)).
\end{itemize}

Since $x_{n+2}\le y_{n+1}\le  x_{n+1}$ and $x_{n+1}-x_{n+2}\le x_{1}-x_{2} $, we get 
$x_1-x_{n+1}\le x_1-y_{n+1}\le  x_1-x_{n+2}\le 2(x_1-x_{n+1})$ and we have
\begin{align*}
I^{(n+1)}_{n+1}&\asymp
\int_{x_{n+2}}^{x_{n+1}}\,e^{-(\lambda_{n+1}-\lambda_{n+2})\,(x_{n+1}-y_{n+1})}
\,\frac{(x_1-y_{n+1})\,(\lambda_1-\lambda_{n+1})}{1+(x_1-y_{n+1})\,(\lambda_1-\lambda_{n+1})}
\\&
\,\prod_{j=2}^{n}\,\frac{(x_j-y_{n+1})\,(\lambda_j-\lambda_{n+1})}{1+(x_j-y_{n+1}})
\,(\lambda_j-\lambda_{n+1})\,dy_{n+1}\\
&\asymp
\frac{(x_1-x_{n+1})(\lambda_1-\lambda_{n+1}) }{1+(x_1-x_{n+1})(\lambda_1-\lambda_{n+1}) }
\,\int_{x_{n+2}}^{x_{n+1}}\,e^{-(\lambda_{n+1}-\lambda_{n+2})\,(x_{n+1}-y_{n+1})}
\\&
\,\prod_{j=2}^{n}\,\frac{(x_j-y_{n+1})\,(\lambda_j-\lambda_{n+1})}{1+(x_j-y_{n+1})
\,(\lambda_j-\lambda_{n+1})}\,dy_{n+1}.
\end{align*}

Hence, noting that $I^{(n+1)}_1(\lambda;X)\cdots I^{(n+1)}_n(\lambda;X)
\asymp  I^{(n)}(\lambda_1,\dots,\lambda_n,\lambda_{n+2};X')$, we have
\begin{align*}
I^{(n+1)}(\lambda;X)
&\asymp
I^{(n)}(\lambda_1,\dots,\lambda_n,\lambda_{n+2};X')\,
\frac{(x_1-x_{n+1})\,(\lambda_1-\lambda_{n+1})}{1+(x_1-x_{n+1})\,(\lambda_1-\lambda_{n+1})}
\\&
\int_{x_{n+2}}^{x_{n+1}}\,e^{-(\lambda_{n+1}-\lambda_{n+2})\,(x_{n+1}-y_{n+1})}
\,\prod_{j=2}^{n}\,\frac{(x_j-y_{n+1})\,(\lambda_j-\lambda_{n+1})}{1+(x_j-y_{n+1})
\,(\lambda_j-\lambda_{n+1})}\,dy_{n+1}.
\end{align*}
Finally,
\begin{align*}
\lefteqn{\int_{x_{n+2}}^{x_{n+1}}\,e^{-(\lambda_{n+1}-\lambda_{n+2})\,(x_{n+1}-y_{n+1})}
\,\prod_{j=2}^{n}\,\frac{(x_j-y_{n+1})\,(\lambda_j-\lambda_{n+1})}{1+(x_j-y_{n+1})
\,(\lambda_j-\lambda_{n+1})}\,dy_{n+1}}\\
&=\frac{
\prod_{k=1}^n\,\int_{x_{k+2}}^{x_{k+1}}\,e^{-(\lambda_{k+1}-\lambda_{n+2})\,(x_{k+1}-y_{k+1})}
\,\prod_{j=1}^{k-1}\,\frac{(x_{j+1}-y_{k+1})\,(\lambda_{j+1}-\lambda_{k+1})}
{1+(x_{j+1}-y_{k+1})\,(\lambda_{j+1}-\lambda_{k+1})}\,dy_{k+1}
}
{
\prod_{k=1}^{n-1}\,\int_{x_{k+2}}^{x_{k+1}}\,e^{-(\lambda_{k+1}-\lambda_{n+2})\,(x_{k+1}-y_{k+1})}
\,\prod_{j=1}^{k-1}\,\frac{(x_{j+1}-y_{k+1})\,(\lambda_{j+1}-\lambda_{k+1})}{1+(x_{j+1}-y_{k+1})\,(\lambda_{j+1}-\lambda_{k+1})}
\,dy_{k+1}
}\\
&=\frac{I^{(n)}(\lambda_2  ,\ldots ,\lambda_{n+1},\lambda_{n+2};x_2,\ldots,x_{n+2})}
{I^{(n-1)}(\lambda_2  ,\ldots ,\lambda_{n},\lambda_{n+2};x_2,\ldots,x_{n+1})}.
\end{align*}
\end{proof}

\begin{remark}
When $n=1$, the result
of Proposition
\ref{recursive}
remains valid if we set $I^{(0)}=1$.
\end{remark}

We now prove our main result.

\begin{proof}[Proof of Theorem \ref{main}]
We use induction on the rank. In the case of $A_1$, we have
\begin{align*}
\psi_\lambda(e^X)
&=e^{\lambda_2\,(x_1+x_2)}\,(x_1-x_2)^{-1}\,\int_{x_2}^{x_1}\,e^{(\lambda_1-\lambda_2)\,y}\,dy\\
&
 =e^{\lambda_2\,(x_1+x_2)}\,(x_1-x_2)^{-1}\,\frac{e^{(\lambda_1-\lambda_2)\,x_1}-e^{(\lambda_1-\lambda_2)\,x_2}}{\lambda_1-\lambda_2}\\
&=e^{\lambda_1\,x_1+\lambda_2\,x_2}\,\frac{1-e^{-(\lambda_1-\lambda_2)\,(x_1-x_2)}}{(\lambda_1-\lambda_2)\,(x_1-x_2)}
\asymp e^{\lambda_1\,x_1+\lambda_2\,x_2}\,\frac{1}{1+(\lambda_1-\lambda_2)\,(x_1-x_2)}
\end{align*}
since $1-e^{-u}\asymp u/(1+u)$ for $u\geq 0$.

Assume that the result is true for $A_r$, $1\leq r\leq n$, $n\geq 1$.  Using \eqref{iter} and the induction hypothesis, we have for $r=1$, \dots, $n+1$
and $\lambda,X$ in positive Weyl chamber in $\R^{r+1}$
\begin{align*}
\lefteqn{\pi(X)\,\pi(\lambda')\,e^{-\lambda(X)}\,\psi_\lambda([x_1,\dots,x_r,x_{r+1}])}\\ 
&=
r! \,\pi(\lambda')\,e^{-\lambda(X)}
\,e^{\lambda_{r+1}\,\sum_{k=1}^{r+1}\,x_k}\,
\int_{x_{r+1}}^{x_r}\cdots\int_{x_2}^{x_1}
\psi_{\lambda_0}(e^Y)\,\prod_{i<j<r+1}\,(y_i-y_j)\,dy_1\cdots dy_{r}\\
&\asymp
\int_{x_{r+1}}^{x_r} \,\int_{x_r}^{x_{r-1}}\cdots\int_{x_3}^{x_2}\,\int_{x_2}^{x_1}
\,e^{-\lambda_0(X'-Y)}\,\prod_{i<j<r+1}\,\frac{(y_i-y_j)\,(\lambda_i-\lambda_j)}
{1+(y_i-y_j)\,(\lambda_i-\lambda_j)}\,dy_1\,dy_2\,\cdots\,dy_r
\end{align*}
where $X'=\diag[x_1,\dots,x_{r}]$ and $\lambda'=[\lambda_1,\dots,\lambda_{r}]$.  Using the notation introduced in Proposition  \ref{iteriter}, we have
\begin{align*}
\pi(X)\,\pi(\lambda') \,e^{-\lambda(X)}\,\psi_{[\lambda_1,\dots,\lambda_{r+1}]}([x_1,\dots,x_{r+1}]) =r!\,
I^{(r)}(\lambda_1,\dots,\lambda_{r+1},x_1,\dots,x_{r+1}).
\end{align*}

Still using the induction hypothesis, we have
\begin{align}
\pi(X)\,\pi(\lambda')\,e^{-\lambda(X)}\,\psi_{[\lambda_1,\dots,\lambda_{r+1}]}([x_1,\dots,x_{r+1}] &=
r!\,I^{(r)}(\lambda_1,\dots,\lambda_{r+1};x_1,\dots,x_{r+1})\nonumber\\
&\asymp 
\frac{\pi(X)\,\pi(\lambda')}{\prod_{i<j\leq r+1}\,(1+(\lambda_i-\lambda_j)\,(x_i-x_j))}\label{iter++}
\end{align}
for $r=1$, \dots, $n$.

It remains to show that \eqref{iter++} holds for $r=n+1$, {\it i.e.}  that
\begin{align*}
I^{(n+1)}(\lambda_1,\dots,\lambda_{n+2};x_1,\dots,x_{n+2}) \asymp 
\frac{\pi(X)\,\pi(\lambda')}{\prod_{i<j\leq n+2}\,(1+(\lambda_i-\lambda_j)\,(x_i-x_j))}.
\end{align*}

It is sufficient to prove the last formula under the hypothesis that
 $\alpha_1(X)\geq {\alpha_{n+1}}(X)$ since the case $\alpha_1(X)\leq {\alpha_{n+1}}(X)$ is symmetric.
Now, according to Proposition \ref{recursive} and \eqref{iter++},
\begin{align*}
I^{(n+1)}(\lambda;X)&\asymp 
\frac{(x_1-x_{n+1})(\lambda_1-\lambda_{n+1}) }{1+(x_1-x_{n+1})(\lambda_1-\lambda_{n+1}) }
\, I^{(n)}(\lambda_1  ,\ldots ,\lambda_n,\lambda_{n+2};x_1,\ldots,x_{n+1})
\\&\qquad
I^{(n)}(\lambda_2  ,\ldots ,\lambda_{n+1},\lambda_{n+2};x_2,\ldots,x_{n+2})
\left(I^{(n-1)}(\lambda_2  ,\ldots ,\lambda_{n},\lambda_{n+2};x_2,\ldots,x_{n+1})\right)^{-1}\\
&\asymp 
\frac{(x_1-x_{n+1})(\lambda_1-\lambda_{n+1}) }{1+(x_1-x_{n+1})(\lambda_1-\lambda_{n+1}) }
\\&\qquad
\frac{\prod_{i<j\leq n+1}\,(x_i-x_j)\,\prod_{i<j< n+1}\,(\lambda_i-\lambda_j)}{\prod_{i<j\leq n}\,(1+(x_i-x_j)\,(\lambda_i-\lambda_j))\,\prod_{i=1}^n\,(1+(x_i-x_{n+1})\,(\lambda_i-\lambda_{n+2}))}
\\&\qquad
\frac{\prod_{1<i<j\leq n+2}\,(x_i-x_j)\,\prod_{1<i<j\leq n+1}\,(\lambda_i-\lambda_j)}{\prod_{1<i<j\leq n+2}\,(1+(x_i-x_j)\,(\lambda_i-\lambda_j))}
\\&\qquad
\,\frac{\prod_{1<i<j\leq n}\,(1+(x_i-x_j)\,(\lambda_i-\lambda_j))\,\prod_{i=2}^n\,(1+(x_i-x_{n+1})\,(\lambda_i-\lambda_{n+2}))}
{\prod_{1<i<j\leq n+1}\,(x_i-x_j)\,\prod_{1<i<j< n+1}\,(\lambda_i-\lambda_j)}\\
&=\frac{x_1-x_{n+1}}{x_1-x_{n+2}}\,\frac{1+(x_1-x_{n+1})\,(\lambda_1-\lambda_2)}{1+(x_1-x_{n+1})\,(\lambda_1-\lambda_2)}
\,\frac{\prod_{i<j\leq n+2}\,(x_i-x_j)\,\prod_{i<j\leq n+1}\,(\lambda_i-\lambda_j)}{\prod_{i<j\leq n+2}\,(1+(\lambda_i-\lambda_j)\,(x_i-x_j))}.
\end{align*}

The result follows since $x_1-x_{n+1}\asymp x_1-x_{n+2}$ given that $x_1-x_2\geq x_{n+1}-x_{n+2}$.
\end{proof}

\section{Comparison with the estimates of Anker et al.{} in \cite{AnkerDH}. Conjecture for Dunkl setting}

In \cite [Theorems 4.1 p.{} 2372 and 4.4, p.{} 2377]{AnkerDH} the following estimates were proven for the heat kernel $p_t(X,Y)$ in the Dunkl setting on $\R^n$.  There exists positive constants 
$c_1$, $c_2$, $C_1$ and $C_2$ such that for all  $X,Y\in \overline{\a^+}$ 
\begin{equation}\label{anker}
 \frac{C_1 e^{-c_1|X-Y|^2/t}}{\min\{
 w(B(X,\sqrt{t})),  w(B(Y,\sqrt{t})) \}}\le p_t(X,Y) \le  \frac{
 C_2 e^{-c_2|X-Y|^2/t}}{\max\{
 w(B(X,\sqrt{t})),  w(B(Y,\sqrt{t})) \}}
\end{equation}
where $w$ is the $W$-invariant reference measure (in our paper $w=\pi(X)^2 dX$)  and
the $w$-volume of a  ball  satisfies the estimate (\cite[p.{} 2365]{AnkerDH})
\begin{align*}
w(B(X,r)) \asymp r^n \prod_{\alpha>0}( r+ \alpha(X) )^{2k(\alpha)}.
\end{align*}
The same  estimates follow for
$p_t^W(X,Y)$.
Our sharp estimates  in Corollary \ref{heatAn}
for $k(\alpha)=1$  in the $W$-radial case $A_n$  suggest that
$c_1=c_2=1/4$ in \eqref{anker} 
and that products of terms $(t+\alpha(X)\alpha(Y))^{k(\alpha)}$ are natural in place of separate 
terms $w(B(X,\sqrt{t}))$ and $  w(B(Y,\sqrt{t}))$.
On the other hand, estimates  \eqref{anker}  and  in Corollary \ref{heatAn}
suggest that the following conjecture is true in the Dunkl setting.
\begin{conj}
The Weyl-invariant heat kernel
for a root system $\Sigma$ in $\R^d$
satisfies the following estimates
\begin{equation}\label{HEATconj}
p_t^W(X,Y)
\asymp t^{-\frac{d}{2}}
\,\frac{e^{\frac{-|X-Y|^2}{4t}}}{\prod_{\alpha>0}\,(t+ \alpha(X)\alpha(Y))^{k(\alpha)}}.
\end{equation}
Formula \eqref{heatSpher} then implies that the $W$-invariant Dunkl kernel satisfies the estimate
\begin{align*}
E_k^W(X,Y) \asymp \frac{e^{\lambda(X)}}{\prod_{\alpha>0}\,(1+\alpha(X)\,\alpha(\lambda))^{k(\alpha)}}.
\end{align*}
\end{conj}
\section{Additional formulas for $p_t^W(X,Y)$}

Let us finish by giving  formulas relating the heat kernel $p_t^W(X,Y)$ with the  spherical functions $\psi_{i\lambda}$ and $\phi_{i\lambda}$.  These formulas can be useful in further study of the kernel $p_t^W(X,Y)$.

\begin{proposition}\label{HEAT-inverseFourier}
\begin{itemize}
\item[(a)]  In the flat Riemannian  symmetric case, the following formula holds:
\begin{equation}\label{heat-inverseFourier}
p_t^W(X,Y)=C\,\int_{\a} e^{-|\lambda|^2t} \,\psi_{i\,\lambda}(X)\,\psi_{-i\,\lambda}(Y) \,\pi(\lambda)^2\,d\lambda,\qquad C>0.
\end{equation}
\item[(b)] In the curved  non-compact Riemannian  symmetric case
the following formula holds
\begin{equation}\label{heat-inverseFourierGeneral}
p_t^W(X,Y)=C\,\int_{\a^*}\, e^{-(|\lambda|^2+|\rho|^2)\,t} \,\phi_{i\,\lambda}(X)\,\phi_{-i\,\lambda}(Y) \,\frac{d\lambda}{|c(\lambda)|^2}
\end{equation}
where $c(\lambda)$ is the Harish-Chandra $c$-function (refer to \cite{Helgason3} for details).
The constant $C$ can be given explicitly.
\end{itemize}
\end{proposition}

\begin{proof} We will prove (b).
We show that the right hand side of equation \eqref{heat-inverseFourier} satisfies the definition of the heat kernel. For a test function $f$, consider
\begin{align*}
u(X,t)=C\,\int_\a\,\int_{\a}\, e^{-(|\lambda|^2+|\rho|^2)\,t}\, \phi_{i\,\lambda}(X)\,\phi_{-i\,\lambda}(Y) \,K\,|c(\lambda)|^{-2}\,d\lambda\,f(Y)\,|c(\lambda)|^{-2}\,dY
\end{align*}
where $K\,|c(\lambda)|^{-2}\,d\lambda$ is Plancherel measure.

The fact that $\Delta\,u(X,t)=\frac{\partial ~}{\partial t}\,u(X,t)$ where $\Delta$ is the radial Laplacian follows easily from the fact that 
$\Delta\,\phi_{i\,\lambda}(X)=-(|\lambda|^2+|\rho|^2)\,\phi_{i\,\lambda}(X)$ and $\frac{\partial ~}{\partial t}\,e^{-(|\lambda|^2+|\rho|^2)]\,t} =-(|\lambda|^2+|\rho|^2)\,e^{-|\lambda|^2t} $.
Now, using Fubini's theorem,
\begin{align*}
u(X,t)&=C\,K\,\int_\a\, e^{-(|\lambda|^2+|\rho|^2)\,t}\,\left[\int_{\a}\,\phi_{-i\,\lambda}(Y) \,f(Y)\,|c(\lambda)|^{-2}\,dY\right]\,\phi_{i\,\lambda}(X)\,\pi(\lambda)^2\,d\lambda\\
&=C\,K\,\int_\a\, e^{-(|\lambda|^2+|\rho|^2)\,t}\,\tilde{f}(\lambda)\,\phi_{i\,\lambda}(X)\,\pi(\lambda)^2\,d\lambda
\end{align*}
which tends to $f(X)$ as $t\to 0$ by the dominated convergence theorem.
\end{proof}

\begin{remark}
The heat kernel estimates
of $h_t^W(X)= p_t^W(X,0)$
on symmetric spaces (\cite{Anker1} and references therein) are based on the inverse spherical Fourier transform formula
which is a special case of \eqref{heat-inverseFourierGeneral} when $Y=0$. Thus one may hope that estimates of $p_t^W(X,Y)$ can be deduced from \eqref{heat-inverseFourierGeneral}.
\end{remark}

\begin{remark}\label{product}
The passage from $h_t^W(X)$ to $ p_t^W(X,Y)$ is well understood at the group level:
\begin{align*} 
p_t^W(g,h)=h_t^W(h^{-1}g),
\end{align*}
 which is equivalent to 
\begin{align*}
p_t^W(X,Y)=\int_K \,h_t^W(e^{-Y}\,k^{-1}\,e^X)\,dk
\end{align*}
and to
\begin{align}\label{heat-product}
 p_t^W(X,Y)= \int_{\a}\, h_t^W(H) \,k(H,-Y,X)\,\pi(H)\,dH,
\end{align}
where the last formula contains the product formula kernel $k$ which is defined by
\begin{align*}
\int_{\a}\, \psi_\lambda(e^H)\,k(H,X,Y)\, \pi(H)\,dH=\psi_\lambda(e^X)\,\psi_\lambda(e^Y)
=\int_K \,\psi_\lambda(e^X \,k \,e^Y) \,dk.
\end{align*}
Similarly,
\begin{align}\label{heat-product-curved}
\tilde p_t^W(X,Y)= \int_{\a}\, \tilde h_t^W(H) \,\tilde k(H,-Y,X)\,\delta(H)\,dH,
\end{align}
where the last formula contains the product formula kernel $\tilde k$ which is defined by
\begin{align*}
\int_{\a}\, \phi_\lambda(e^H)\,\tilde k(H,X,Y)\, \delta(H)\,dH=\phi_\lambda(e^X)\,\phi_\lambda(e^Y)
=\int_K \,\phi_\lambda(e^X \,k \,e^Y) \,dk.
\end{align*}

\end{remark}

\end{document}